\documentclass[12pt]{amsart}

\usepackage{amsmath,amsthm,amsfonts,latexsym,epsfig,fullpage}

\newtheorem{thm}{Theorem}[section]
\newtheorem{lem}[thm]{Lemma}
\newtheorem{pro}[thm]{Proposition}
\newtheorem{exa}[thm]{Example}

\newtheorem{cor}[thm]{Corollary}

\newcommand{\Aut}{{\mathrm{Aut}}\;}

\newcommand{\al}{\alpha}
\newcommand{\be}{\beta}
\newcommand{\ga}{\gamma}

\newcommand{\om}{\omega}

\newcommand{\Ta}{\Theta}
\newcommand{\Tap}{\Theta^{\perp}}
\newcommand{\La}{\Lambda}
\newcommand{\de}{{\delta}}
\newcommand{\la}{{\lambda}}
\newcommand{\Ga}{\Gamma}
\newcommand{\Up}{\Upsilon}
\newcommand{\ph}{\widehat{p}}
\newcommand{\eh}{\widehat{e}}
\newcommand{\hh}{\widehat{h}}
\newcommand{\sh}{\widehat{s}}
\newcommand{\xh}{\widehat{x}}
\newcommand{\Bh}{\widehat{B}}

\newcommand{\vep}{\varepsilon}

\newcommand{\ld}{\ldots}
\newcommand{\cd}{\cdots}

\newcommand{\bp}{\mathbf{p}}
\newcommand{\bq}{\mathbf{q}}
\newcommand{\bze}{\mathbf{0}}
\newcommand{\bph}{\mathbf{{\widehat{p}}}}

\newcommand{\cC}{\mathcal{C}}
\newcommand{\cP}{\mathcal{P}}

\newcommand{\cR}{\mathcal{R}}

\newcommand{\fS}{\mathfrak{S}}
\newcommand{\bbQ}{\mathbb{Q}}
\newcommand{\bbZ}{\mathbb{Z}}

\newcommand{\del}{\partial}

\newcommand{\dmrjdel}[1]{}

\title{Symmetric functions, codes of partitions and the KP hierarchy}
\author[S.~R.~Carrell]{S.~R.~Carrell$^*$}
\author[I.~P.~Goulden]{I.~P.~Goulden$^\dagger$}

\thanks{
${\hspace{-1ex}}^*$Department of Combinatorics and Optimization,
                                                University of Waterloo, Waterloo, Ontario, Canada.;    \\
${\hspace{.35cm}}$ \texttt{srcarrell@uwaterloo.ca}}

\thanks{
${\hspace{-1ex}}^\dagger$Department of Combinatorics and Optimization,
                                                University of Waterloo, Waterloo, Ontario, Canada.;  \\
${\hspace{.35cm}}$ \texttt{ipgoulden@uwaterloo.ca}}

\date{February 25, 2009}

\begin{document}
\maketitle

\begin{abstract}
We consider an operator of Bernstein for symmetric functions, and give an explicit
formula for its action on an arbitrary Schur function. This formula is given
in a remarkably simple form when written in terms of some notation
based on the code of a partition. As an
application, we give a new and very simple proof of a classical result for
the KP hierarchy, which involves the Pl\"ucker relations for Schur function coefficients
in a $\tau$-function for the hierarchy. This proof is especially compact because
of a restatement that we give for the Pl\"ucker relations that is symmetrical in
terms of partition code notation.
\end{abstract}

%%%%%%%%%%%%%%%%%%%%%%%%%%%%%%%%%%%%%%%%%%%%%%%%

\section{Introduction}\label{sec1}

In Macdonald's fundamental book~\cite[p.~95]{mac} on symmetric functions, he considers 
the operators $B(t)$, $B_n$ given by
\begin{equation}\label{Bernp}
B(t):=\sum_{n\in\bbZ}B_nt^n:=\exp\left( \sum_{k \geq 1} \frac{t^k}{k} p_k\right)
\exp\left( -\sum_{k \geq 1} t^{-k} \frac{\del}{\del p_k}\right),
\end{equation}
which he attributes to Bernstein~\cite[p.~69]{z}.
Here $t$ is an indeterminate, and $p_k$ is the $k$th power sum symmetric function
of a countable set of indeterminates (in which the power sums are symmetric). 
Of course, since the power sum symmetric functions $p_1,p_2,\ldots$ in
a countable set of indeterminates are algebraically independent (see Section~\ref{sec4} for
this and other basic results about symmetric functions), we can regard $p_1,p_2,\ldots$ as
indeterminates.  Hence we consider Bernstein's operators $B(t)$, $B_n$ applied to
formal power series in $p_1,p_2,\ldots$. Let the \emph{weight} of the
monomial $p_1^{i_1}p_2^{i_2}\cdots$ be given by $i_1+2i_2+\ldots$. Then note that as a series in $t$, $B(t)$ goes
to infinity in both directions, but that when it is applied to a monomial of
weight $m$ in $p_1,p_2,\ldots$, the result is a Laurent series with minimum degree $-m$ in $t$.

The main result of this paper is an explicit expansion for the Laurent series that
results from applying $B(t)$ to an arbitrary Schur symmetric function (where the Schur function
is written as a polynomial in $p_1,p_2,\ldots$ with symmetric group characters
as coefficients, as given in~\eqref{psandsp}).
This result, appearing as Theorem~\ref{Bsla}, is given in a remarkably simple
form by using an indexing notation for partitions that is natural in
terms of the \emph{code} of the partition. 

For an application of our main result, we turn to the KP hierarchy of mathematical
physics. By combining our main result and its dual
form, we are able to give a  new, simple proof (in Theorem~\ref{SchPlu}) of
the classical result for the KP hierarchy that relates
the Schur symmetric function expansion of a $\tau$-function for the hierarchy
and the Pl\"ucker relations 
for the coefficients in this expansion.
This proof is made especially compact by our restatement (as Theorem~\ref{sympluck}) of
the Pl\"ucker relations,
in a highly symmetrical form
in terms of the code notation.

The KP (Kadomtsev-Petviashvili) hierarchy is a completely integrable hierarchy
that generalizes the KdV hierarchy (an integrable hierarchy is a family of partial differential
equations which are simultaneously solved;
for a comprehensive and concise account of
the KP hierarchy with an emphasis on the view of physics, see Miwa, Jimbo and Date~\cite{mjd}).
Over the last two
decades, there has been strong interest in the relationship between integrable
hierarchies and moduli spaces of curves. This began with Witten's Conjecture~\cite{w} for
the KdV equations, proved by Kontsevich~\cite{ko} (and more recently by a number of
others, including~\cite{kl}). Pandharipande~\cite{p} conjectured that solutions to
the Toda equations arose in a related context, which was proved by Okounkov~\cite{o}, who
also proved the more general result that a generating series for what he
called \emph{double Hurwitz numbers} satisfies the KP hierarchy.  Kazarian
and Lando~\cite{kl}, in their recent proof of Witten's conjecture, showed that
it is implied by Okounkov's result for double Hurwitz numbers. More
recently, Kazarian~\cite{ka} has given a number of
interesting results
about the structure of solutions to the KP hierarchy.
Combinatorial aspects of this connection have been studied by Goulden and Jackson~\cite{gj}.

The structure of this paper is as follows.
In Section~\ref{sec2}, we describe the notation for partitions (which index symmetric
functions), their diagrams and codes (in terms of the physics presentation
in~\cite{mjd}, the code is
equivalent to a \emph{Maya diagram} without the location parameter of
charge). In Section~\ref{sec4} we describe the algebra of symmetric functions,
and give our main result, with a combinatorial proof in terms of the  diagrams of partitions.
Section~\ref{sec3} contains our symmetrical restatement
of the Pl\"ucker relations.
Finally, in Section~\ref{sec5}, we give our new proof of the Schur-Pl\"ucker result for
the KP hierarchy, together with a description of an equivalent system of
partial differential equations for the KP hierarchy in terms of codes of partitions.

\section{Partitions and codes}\label{sec2}

We begin with some notation for partitions (see \cite{mac}, \cite{s}). 
If  $\la_1,\ld ,\la_n$ are integers with $\la_1\ge \cd\ge\la_n\ge 1$ and $\la_1+\cd +\la_n=d$,  then $\la=(\la_1,\ld ,\la_n)$ is said to be a \emph{partition} of $|\la|:=d$ (indicated by writing $\la\vdash d$) with $l(\la):=n$ \emph{parts}. The empty list $\vep$ of integers is to be regarded as a partition of $d=0$ with $n=0$ parts, and let $\cP$ denote the set of all partitions. If $\la$ has $f_j$ parts equal to $j$ for $j=1,\ld ,d$, then we also write $\la=(d^{f_d},\ld ,1^{f_1})$, where convenient. Also, $\Aut\la$ denotes the set of permutations of the $n$ positions that fix $\la$; therefore  $|\Aut\la |=\prod_{j\ge 1}f_j!$.
The \emph{conjugate} of $\la$ is the partition $\la'=(\la'_1,\ld ,\la'_m)$,
where $m=\la_1$, and $\la'_i$ is equal to the number of
parts $\la_j$ of $\la$ with $\la_j\ge i$, $i=1,\ld ,m$.
The \emph{diagram} of $\la$ is a left-justified array
of juxtaposed unit squares, with $\la_i$ squares in the $i$th row. For example, the
diagram of the partition $(6,5,5,4,1)$ is given on the left hand side
of Figure~\ref{diagbdy}. The diagram gives a compact geometrical description of the conjugate; the
diagram of $\la'$ is obtained by reflecting the diagram of $\la$ about the main diagonal.
\begin{figure}[ht]
\begin{center}
\scalebox{.60}{\includegraphics{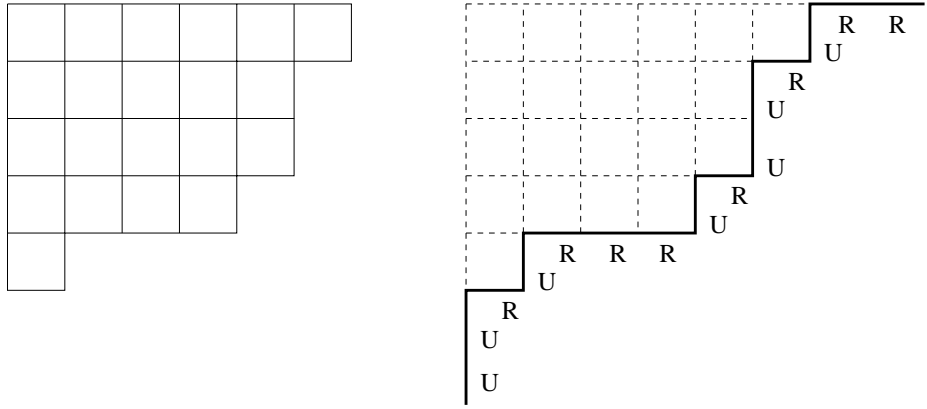}}
\end{center} 
\caption{The diagram and code for the partition $(6,5,5,4,1)$.}\label{diagbdy}
\end{figure}

The \emph{code} of a partition is a two-way infinite binary string in
two symbols, say U, R. The string is infinitely U to the left, and infinitely R to the right, and
represents a lattice path with unit up-steps (for "U") and unit right-steps (for "R"); the path starts
moving up the $y$-axis towards the origin, and ends moving out the $x$-axis away from the origin, and
forms the lower right hand boundary of the diagram of the partition, when the diagram is
placed with its upper left hand corner at the origin. An example
is given on the right hand side
of Figure~\ref{diagbdy}, which gives us the string ...UURURRRURUURURR... as
the code of the partition $(6,5,5,4,1)$ (see, e.g.,~\cite[p. 467]{s} for more on codes;
in~\cite{s}, the
symbols in the code are $0$, $1$, but we prefer U, R as more mnemonic.)

We let $\la^{(i)}$, $i\ge 1$, be the partition whose code is obtained from the code of
the partition $\la$ by switching the $i$th R (from the left) to U. If $\la=(\la_1,\ld ,\la_n)$, then
we immediately have
\begin{equation}\label{formisw}
\la^{(i)}=(\la_1-1,\ld ,\la_j-1,i-1,\la_{j+1},\ld \la_n),
\end{equation}
where $j$ is chosen (uniquely) from $0,\ld ,n$ so that $\la_j\ge i>\la_{j+1}$ (with the conventions
that $\la_{n+1}=0$ and $\la_0=\infty$). Now define $u_i(\la)$ to be the
number of up-steps U that
follow the $i$th right-step R from the left in the code of $\la$. Then note
that $u_i(\la)=j$, and that we have $|\la^{(i)}|=|\la|-j+i-1$ from~\eqref{formisw}, so
we can determine $u_i(\la)=j$ in terms of $i$ via
\begin{equation}\label{ij}
u_i(\la)=|\la|-|\la^{(i)}|+i-1.
\end{equation} 
Also note that $u_i(\la)$ weakly
decreases as $i$ increases, so we also obtain
\begin{equation}\label{laiinc}
|\la|-l(\la)=|\la^{(1)}|<|\la^{(2)}|<\cdots,\qquad
|\la^{(i)}|=|\la|+i-1,\quad i>\la_1.
\end{equation}

We also let $\la^{(-i)}$, $i\ge 1$, be the partition whose code is
obtained from the code of
the partition $\la$ by switching the $i$th U (from the right) to R. If $\la=(\la_1,\la_2,\ld )$,
with $\la_1\ge\la_2\ld\ge 0$ (\emph{i.e.}, $\la$ has a finite number, say $n$, of positive parts,
and then append an infinite sequence of $0$'s), then
we immediately have
\begin{equation}\label{formminisw}
\la^{(-i)}=(\la_1+1,\ld ,\la_{i-1}+1,\la_{i+1},\ld ).
\end{equation}
Thus we have $|\la^{(-i)}|=|\la|-\la_i+i-1$, and since $\la_i$  weakly decreases as $i$ increases,
we obtain
\begin{equation}\label{laminiinc}
|\la|-\la_1=|\la^{(-1)}|<|\la^{(-2)}|<\cdots,\qquad
|\la^{(-i)}|=|\la|+i-1,\quad i>l(\la).
\end{equation}

Finally, we define  skew diagrams. For partitions $\la=(\la_1,\ld ,\la_n)$ and $\mu=(\mu_1,\ld ,\mu_m)$,
with $n\ge m$ and $\la_i\ge\mu_i$ for $i=1,\ld ,m$, the \emph{skew} diagram $\la -\mu$ is
the array of unit squares obtained by removing the squares of the diagram of $\mu$ from
the diagram of $\la$. For example, the skew diagram $(6,5,5,4,1)-(3,3,2,1)$ is given
in Figure~\ref{skewdiag}.
\begin{figure}[ht]
\begin{center}
\scalebox{.60}{\includegraphics{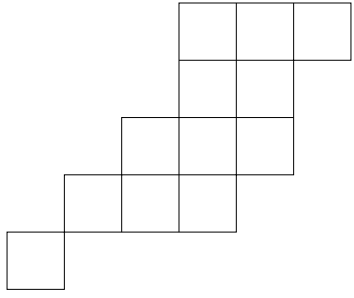}}
\end{center}
\caption{The skew diagram $(6,5,5,4,1)-(3,3,2,1)$.}\label{skewdiag}
\end{figure}

A skew diagram is a \emph{horizontal strip} if it contains at most one square
in every column; it is a \emph{vertical strip} if it contains at most one square 
in every row. When the total number of squares in a horizontal (or vertical) strip
is equal to $k$, we may refer to this skew diagram as a horizontal (or vertical) $k$-strip.

\section{Codes of partitions and symmetric functions}\label{sec4}

\subsection{The algebra of symmetric functions}

We consider symmetric functions in $x_1,x_2,\ld$, and refer to~\cite{s},~\cite{mac} for the results that we use in this paper.
The $i$th \emph{power sum} symmetric function is  $p_i=\sum_{j\ge 1}x_j^i$, for $i\ge 1$, with  $p_0:=1$. The  $i$th \emph{complete} symmetric function $h_i$ and the $i$th \emph{elementary} symmetric function are defined by
$\sum_{i\ge 0}h_it^i=\prod_{j\ge 1}(1-x_jt)^{-1}$ and $\sum_{i\ge 0}e_it^i=\prod_{j\ge 1}(1+x_jt)$, respectively,
and are related to the power sums through
\begin{equation}\label{hewithp}
\sum_{i\ge 0}h_it^i=\exp \sum_{k\ge 1}\frac{p_k}{k} t^k,
\qquad\qquad
\sum_{i\ge 0}e_it^i=\exp \sum_{k\ge 1}\frac{p_k}{k}(-1)^{k-1} t^k.
\end{equation}
In the algebra $\La_{\bbQ}$ of symmetric functions in $x_1,x_2,\ld$ over $\bbQ$, the $p_i$, $i\ge 1$ are algebraically independent,
the $h_i$, $i\ge 1$ are algebraically independent, and the $e_i$, $i\ge 1$ are algebraically independent. This 
algebra is \emph{graded}, by total degree in the underlying indeterminates $x_1,x_2,\ld$.  Moreover, $\La_{\bbQ}$ is
generated by the $p_i$'s, the $h_i$'s or the $e_i$'s, so we have
$$\La_{\bbQ}=\bbQ [p_1,p_2,\ld ]=\bbQ [h_1,h_2,\ld ]=\bbQ [e_1,e_2,\ld ].$$

Now suppose we define quantities $p_{\la}$, $h_{\la}$, $e_{\la}$ indexed by the partition $\la=(\la_1,\ld ,\la_n)$, \emph{multiplicatively},
which means that we write $p_{\la}=p_{\la_1}\cd p_{\la_n}$, $h_{\la}=h_{\la_1}\cd h_{\la_n}$, $e_{\la}=e_{\la_1}\cd e_{\la_n}$, with
the empty product convention $p_{\vep}=h_{\vep}=e_{\vep}=1$. Then of course, regarded as a vector space over $\bbQ$, $\La_{\bbQ}$ has
(multiplicative) bases $\{ p_{\la}:\la\in\cP\}$, $\{ h_{\la}:\la\in\cP\}$ and $\{ e_{\la}:\la\in\cP\}$.

Another (non-multiplicative)
basis is $\{ s_{\la}:\la\in\cP\}$, where $s_{\la}$ is the \emph{Schur} function of \emph{shape} $\la$. There
are many fascinating aspects about
Schur functions, but the one that we shall focus on in this paper is their connection with 
characters of the symmetric group
(we refer to~\cite{s} for the results about characters and representations of the symmetric group used in this paper).
The \textit{conjugacy classes} of the symmetric group $\fS_d$ on $\{ 1,\ld ,d\}$ are
indexed by the partitions of $d$ (and for partition $\la=(\la_1,\ld,\la_n)$, the conjugacy
class $\cC_{\la}$ consists of all permutations in $\fS_d$ whose disjoint cycle lengths
give the parts of $\la$ in some order).  The \textit{irreducible representations} of $\fS_d$ are
also indexed by the partitions of $d$.
For partitions $\la$, $\mu$ of $d$, let $\chi^{\la}_{\mu}$ denote
the \emph{character} of the irreducible representation of $\fS_d$ indexed by $\la$, evaluated
at any element of the conjugacy class $\cC_{\mu}$ (in general, characters are constant on conjugacy classes);
 we usually refer to $\chi^{\la}_{\mu}$ generically as
an irreducible character.
Then, to change bases of $\La_{\bbQ}$ between $\{ p_{\la}:\la\in\cP\}$, and $\{ s_{\la}:\la\in\cP\}$,
we have 
\begin{equation}\label{psandsp}
p_{\la}=\sum_{\mu\vdash d}\chi^{\mu}_{\la}s_{\mu},
\qquad\qquad   s_{\la}=\sum_{\mu\vdash d} \frac{|\cC_{\mu}|}{d!}\chi^{\la}_{\mu}p_{\mu}, \qquad \la\vdash d.
\end{equation}

The endomorphism $\om:\La_{\bbQ}\rightarrow\La_{\bbQ}$ defined by $\om(e_n)=h_n$, $n\ge 1$ is an involution on $\La_{\bbQ}$.
For Schur functions, we have
\begin{equation}\label{sconj}
\om(s_{\la})=s_{\la'}.
\end{equation}

We have various multiplication rules for Schur functions. For example,
\begin{equation}\label{hmult}
h_ns_{\la}=\sum_{\mu}s_{\mu},
\end{equation}
where the sum is over partitions $\mu$ such that $\mu-\la$ is a horizontal $n$-strip. In general, when
the involution $\om$ is applied to an equation for symmetric functions, we call the
resulting equation the \emph{dual}. For example, applying $\om$ to~\eqref{hmult}, we obtain the
dual result
\begin{equation}\label{emult}
e_ns_{\la}=\sum_{\mu}s_{\mu},
\end{equation}
where the sum is over partitions $\mu$ such that $\mu-\la$ is a
vertical $n$-strip. (Equations~\eqref{hmult} and~\eqref{emult} are often referred to as \emph{Pieri} rules.)

Let $\langle\; ,\;\rangle$ be the bilinear form for $\La_{\bbQ}$ defined by
\begin{equation}\label{Schbasis}
\langle s_{\la},s_{\mu}\rangle =\de_{\la ,\mu},\qquad \la,\mu\in\cP .
\end{equation}
For each symmetric function $f\in\La_{\bbQ}$, let $f^{\perp}:\La_{\bbQ}\rightarrow\La_{\bbQ}$ be the
adjoint of multiplication by $f$, so
$$\langle f^{\perp}g_1,g_2\rangle =\langle g_1,fg_2\rangle,\qquad g_1,g_2\in\La_{\bbQ}.$$
Then 
\begin{equation}\label{pnperp}
p_n^{\perp}=n\, \frac{\partial}{\partial p_n}, \qquad n\ge 1.
\end{equation}
Moreover, from~\eqref{hmult} and~\eqref{Schbasis} we have
\begin{equation}\label{hperp}
h_n^{\perp}s_{\la}=\sum_{\mu} s_{\mu},
\end{equation}
where the sum is over partitions $\mu$ such that $\la-\mu$ is a horizontal $n$-strip.
Similarly, from~\eqref{emult} and~\eqref{Schbasis} we have
\begin{equation}\label{eperp}
e_n^{\perp}s_{\la}=\sum_{\mu} s_{\mu},
\end{equation}
where the sum is over partitions $\mu$ such that $\la-\mu$ is a vertical $n$-strip.

We now consider Bernstein's symmetric function operators $B(t)$, $B_n$ defined in~\eqref{Bernp}.
Equivalently, from~\eqref{hewithp} and~\eqref{pnperp}, together with
the fact that the partial differential operators $\frac{\del}{\del p_k}$ commute, we have
\begin{equation}\label{Tadef}
B(t):=\sum_{n\in\bbZ}B_nt^n :=\sum_{k,m\ge 0}(-1)^m t^{k-m}h_k\, e_m^{\perp},
\end{equation}
(where $B_n$ is an operator on $\La_{\bbQ}$ for each $n\in\bbZ$).
Then we also have
\begin{equation}\label{Taperpdef}
B^{\perp}(t)=\sum_{n\in\bbZ}B^{\perp}_nt^n =\sum_{k,m\ge 0}(-1)^m t^{k-m}e_m\, h_k^{\perp},
\end{equation}
and we immediately obtain $B^{\perp}_n=(-1)^n\om B_{-n} \om$ (note that this corrects
the relation given in~\cite[p. 96]{mac}). In terms of the series $B$ itself, this becomes
the operator equation
\begin{equation}\label{BBperp}
B^{\perp}(t)=\om B(-t^{-1}) \om,
\end{equation}
where $\om$ applied to a series in $t$ acts linearly, on the coefficients of each monomial in $t$.

\subsection{A combinatorial treatment of Bernstein's operator}

We define $\cR_{\la ,\mu}$ to be the set of partitions $\nu$ such that $\la -\nu$ is
a vertical strip, and $\mu -\nu$ is a horizontal strip. Let
\begin{equation}\label{Bdef}
R_{\la ,\mu}=\sum_{\nu\in\cR_{\la ,\mu}}(-1)^{|\la |-|\nu |},
\end{equation}
which is $0$ when the set $\cR_{\la ,\mu}$ is empty.
This sum arises naturally in the action of $B(t)$ on a Schur function $s_{\la}$, as
given in the following result.

\begin{pro}\label{Tasla}
For any partition $\la$, we have
$$B(t)\, s_{\la} = \sum_{\mu\in\cP}R_{\la ,\mu}\, t^{|\mu |-|\la |} s_{\mu}.$$

\end{pro}

\begin{proof}
The result follows immediately from~(\ref{Tadef}),~\eqref{hmult} and~\eqref{eperp}.
\end{proof}

\begin{figure}[ht]
\begin{center}
\scalebox{.70}{\includegraphics{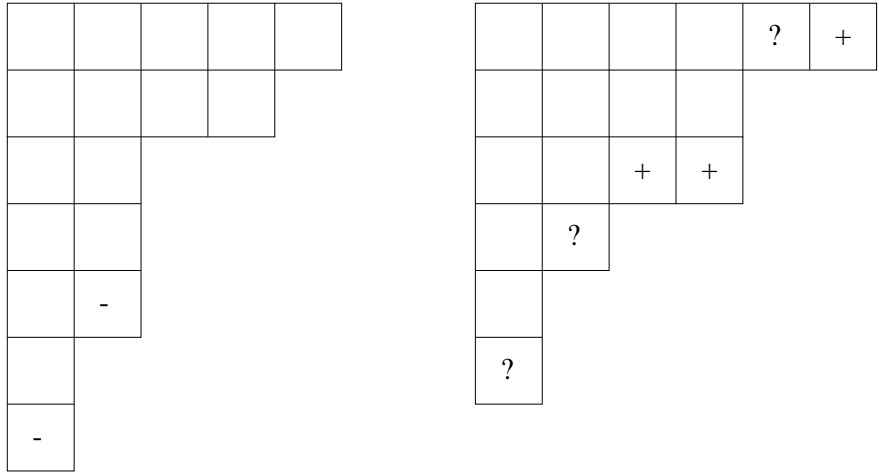}}
\end{center}
\caption{Partitions $\la$ and $\mu$ with $\la$-ambiguous squares.}\label{ambig}
\end{figure}

Now we consider the structure of the set $\cR_{\la ,\mu}$, beginning with
the example given in Figure~\ref{ambig}, where the
diagrams of $\la =(5,4,2,2,2,1,1)$ and $\mu =(6,4,4,2,1,1)$ are given on the left
and right of the diagram, respectively. There are three classes of squares
that we shall consider when $\cR_{\la ,\mu}$ is nonempty:
\begin{itemize}
\item
the squares of $\mu$ that are not contained in $\la$, necessarily bottommost in their column
of $\mu$. None of these is contained in any $\nu$ in $\cR_{\la ,\mu}$ (in Figure~\ref{ambig}, these squares
contain ``+''). Such squares are contained in the horizontal strip that is added in the multiplication by an $h_k$;
\item
the squares of $\la$ that are not contained in $\mu$, necessarily rightmost in their row
of $\la$. All of these are contained in every $\nu$ in $\cR_{\la ,\mu}$ (in Figure~\ref{ambig},
these squares contain ``-'')
Such squares are contained in the vertical strip that is deleted in the application of an $e_m^{\perp}$;
\item
the squares that are contained in both $\la$ and $\mu$, that are rightmost in their row
of $\la$, and bottommost in their column of $\mu$. Each of these is contained in some of
the $\nu$ in $\cR_{\la ,\mu}$ (actually in half of the $\nu$), but not others. We call such
a square a $\la$-\emph{ambiguous} square of $\mu$ (in Figure~\ref{ambig},
the $\la$-ambiguous squares of $\mu$ contain ``?''). Such squares may have been contained in both a deleted
vertical strip and an added horizontal strip, or in neither.
\end{itemize}

\begin{lem}\label{ambvan}
If $\mu$ has any $\la$-ambiguous squares, then
$$R_{\la ,\mu} =0.$$
\end{lem}

\begin{proof}
If $\mu$ has any $\la$-ambiguous squares, let $c$ be the rightmost of these (there is
at most one $\la$-ambiguous square in any column of $\mu$, since it can only occur
as the bottommost element of a column). Define the mapping
$$\phi :\cR_{\la ,\mu}\rightarrow\cR_{\la ,\mu}: \nu\mapsto\nu^{\prime}$$
as follows: if $\nu$ contains $c$, then $\nu^{\prime}$ is obtained by removing $c$ from $\nu$;
if $\nu$ does not contain $c$, then $\nu^{\prime}$ is obtained by adding $c$ to $\nu$. This is
well-defined, for the following reasons: since $c$ is rightmost in its row of $\la$ and bottommost
in its column of $\mu$, every square of $\la$ in the same column as $c$ and below $c$ must
belong to the vertical strip $\la -\nu$ (and no other squares in this column can belong
to this vertical strip), so $\la-\nu^{\prime}$ is a vertical strip whether $\nu$ contains $c$ or
not; also, every square of $\mu$ in the same row as $c$ and to the right of $c$ must
belong to the horizontal strip $\mu -\nu$ (and no other squares in this row can belong
to this horizontal strip), so $\mu -\nu^{\prime}$ is a horizontal strip
whether $\nu$ contains $c$ or not).

Clearly $\phi$ is an involution on $\cR_{\la ,\mu}$, so it is a bijection, and thus we have
\begin{equation}\label{Bzero}
R_{\la ,\mu}=\sum_{\nu\in\cR_{\la ,\mu}}(-1)^{|\la |-|\nu |}
=-\sum_{\nu^{\prime}\in\cR_{\la ,\mu}}(-1)^{|\la |-|\nu^{\prime}|}
=-R_{\la ,\mu},
\end{equation}
where, for the second equality, we have changed summation variables to $\phi(\nu )=\nu^{\prime}$.
The result follows 
immediately. (In the context of~(\ref{Bzero}), $\phi$ is often referred
to as a \emph{sign-reversing involution}.)
\end{proof}

If $\cR_{\la ,\mu}$ is nonempty and $\mu$ has no $\la$-ambiguous squares, we
call $\mu$ a $\la$-\emph{survivor}. Note that in this case there is a
unique $\nu$ in $\cR_{\la ,\mu}$; the terminology is chosen since $\mu$ ``survives'' the involution
in Lemma~\ref{ambvan}.

\begin{pro}\label{survchar}
In a $\la$-survivor $\mu$:
\begin{itemize}
\item[(a)] if the rightmost square of row $i$ of $\la$ is not contained in $\mu$, then the
rightmost square of row $i-1$ of $\la$ is not contained in $\mu$; 
\item[(b)] if the bottommost square of column $i$ of $\mu$ is not contained in $\la$, then the
bottommost square of column $i-1$ of $\mu$ is not contained in $\la$.
\end{itemize}
\end{pro}

\begin{proof}
For part~(a), if the rightmost square of row $i$ of $\la$ is not contained in $\mu$ and
the rightmost square of row $i-1$ of $\la$ is contained in $\mu$, then the latter must
be bottommost in its column of $\mu$. But that makes it a $\la$-ambiguous square, impossible
in a $\la$-survivor. For~(b),
if the bottommost square of column $i$ of $\mu$ is not contained in $\la$ and the
bottommost square of column $i-1$ of $\mu$ is contained in $\la$, then the latter is
by definition a $\la$-ambiguous square, impossible in a $\la$-survivor. 
\end{proof}

\subsection{The main result}

Now we are able to determine explicitly the action of $B(t)$ on a single Schur function $s_{\la}$.
The following is our main result.

\begin{thm}\label{Bsla}
For any partition $\la$, we have
\begin{equation*}
B(t)\, s_{\la}=\sum_{i\ge 1}(-1)^{|\la |-|\la^{(i)}|+i-1} t^{|\la^{(i)}|-|\la |}s_{\la^{(i)}}.
\end{equation*}
\end{thm}

\begin{proof}
{}From Proposition~\ref{Tasla} and Lemma~\ref{ambvan}, we have
\begin{equation}\label{intform}
B(t)\, s_{\la} = \sum_{\mu} R_{\la ,\mu}\, t^{|\mu |-|\la |} s_{\mu},
\end{equation}
where the summation is over all $\la$-survivors $\mu$.
Now we characterize the $\la$-survivors.
Suppose $\la=(\la_1,\ld ,\la_n)$, where $\la_1\ge \ld\ge\la_n\ge 1$, and
we let $\la_0=\infty$, $\la_{n+1}=0$. Then in a $\la$-survivor $\mu$, from
Proposition~\ref{survchar}(a), the rightmost cells in
rows $1,\ld ,j$ of $\la$ are not contained in $\mu$, and the rightmost cells
of rows $j+1,\ld ,n$ are contained in $\mu$, for some $j=0,\ld ,n$ with $\la_j>\la_{j+1}$.
But, in order to avoid the
bottommost cell of column $\la_{j+1}$ in $\mu$ being $\la$-ambiguous, then the
bottommost cell of column $\la_{j+1}$ in $\mu$ cannot be contained in $\la$. Thus
we conclude from Proposition~\ref{survchar}(b) that the bottommost cells
in columns $1,\ld ,\la_{j+1}$ of $\mu$ are not contained in $\la$. Also, the bottommost
squares in columns $\la_{j+1},\ld ,i-1$ of $\mu$ are not contained in $\la$ for
some $\la_{j+1}<i\le\la_j$. Finally, for each $i\ge 1$, there exists a choice
of $j=0,\ld ,n$ for which $\la_{j+1}<i\le\la_j$, so the $\la$-survivor $\mu$ described
above exists for each $i\ge 1$. This partition $\mu$ is obtained from $\la$ by deleting
the column strip consisting of the rightmost squares in rows $1,\ld , j$, and
adding the horizontal strip consisting of the bottommost cells in columns $1,\ld ,i-1$.
This gives
$$\mu=(\la_1-1,\ld ,\la_j-1,i-1,\la_{j+1},\ld ,\la_n)=\la^{(i)},$$
where the second equality is from~(\ref{formisw}), and where $j=u_i(\la )$. But we
have $R_{\la ,\la^{(i)}}=(-1)^{j}$, and
the result follows from~\eqref{ij} and~\eqref{intform}.
\end{proof}

Note that the right hand side of the above result is a Laurent series in $t$ for each $\la$,
with minimum power of $t$ given by $t^{|\la^{(1)}|-|\la|}=t^{-l(\la )}$.

The following pair of dual corollaries to our main result will be particularly convenient for
dealing with the KP hierarchy.

\begin{cor}\label{Bsumsla}
For scalars $a_{\la}$, $\la\in\cP$, we have
\begin{equation*}
B(t)\, \sum_{\la\in\cP}a_{\la}s_{\la}
=\sum_{\be\in\cP}s_{\be} \sum_{k\ge 1}(-1)^{k-1} t^{|\be |-|\be^{(-k)}|}a_{\be^{(-k)}}.
\end{equation*}
\end{cor}

\begin{proof}
{}From Theorem~\ref{Bsla} and~\eqref{ij}, we immediately obtain
\begin{equation}\label{simpsum}
B(t)\, \sum_{\la\in\cP}a_{\la}s_{\la}=
\sum_{\la\in\cP}a_{\la}\sum_{i\ge 1}(-1)^{u_i(\la )} t^{|\la^{(i)}|-|\la |}s_{\la^{(i)}}.
\end{equation}
Now from the code description, it is immediate that $\be =\la^{(i)}$ is equivalent
to $\la =\be^{(-k)}$, where $k=u_i(\la )+1$. The result follows immediately, by 
changing summation variables in~\eqref{simpsum} from $\la\in\cP$, $i\ge 1$ to $\be\in\cP$, $k\ge 1$.
\end{proof}

\begin{cor}\label{Bperpsumsla}
For scalars $a_{\la}$, $\la\in\cP$, we have
\begin{equation*}
B^{\perp}(t^{-1})\, \sum_{\la\in\cP}a_{\la}s_{\la}
=\sum_{\al\in\cP}s_{\al} \sum_{m\ge 1}(-1)^{|\al |-|\al^{(m)}|+m-1} t^{|\al |-|\al^{(m)}|}a_{\al^{(m)}}.
\end{equation*}
\end{cor}

\begin{proof}
{}From Theorem~\ref{Bsla},~\eqref{BBperp} and~\eqref{sconj}, we obtain
\begin{equation*}
B^{\perp}(t^{-1})\, \sum_{\la\in\cP}a_{\la}s_{\la}= \sum_{\la\in\cP}a_{\la} \om B(-t)\, s_{\la'}
=\sum_{\la\in\cP}a_{\la}\sum_{i\ge 1}(-1)^{i-1} t^{|(\la')^{(i)}|-|\la'|}s_{((\la')^{(i)})'}.
\end{equation*}
But from the code description, it is immediate that $(\la')^{(i)}=(\la^{(-i)})'$, and
since $|\mu'|=|\mu|$ for any partition $\mu$, we can simplify the double summation above to obtain
\begin{equation}\label{simpsumperp}
B^{\perp}(t^{-1})\, \sum_{\la\in\cP}a_{\la}s_{\la}=
\sum_{\la\in\cP}a_{\la}\sum_{i\ge 1}(-1)^{i-1} t^{|\la^{(-i)}|-|\la|}s_{\la^{(-i)}}.
\end{equation}
As in the proof of Corollary~\ref{Bsumsla}, we have that $\al =\la^{(-i)}$ is equivalent
to $\la =\al^{(m)}$, where $i=u_m(\al )+1$. The result now follows immediately, by
changing summation variables in~\eqref{simpsum} from $\la\in\cP$, $i\ge 1$ to $\al\in\cP$, $m\ge 1$,
and applying~\eqref{ij} to evaluate~$u_m(\al )$ (which is the exponent of $(-1)$  when the summation
is expressed in terms of $\al$, $m$).
\end{proof}

Among the results in~\cite{mac} and~\cite{z} for Bernstein's operators is
$$B_{\la_1}\, \cdots B_{\la_n}\, 1 = s_{\la},$$
where $\la =(\la_1,\ldots ,\la_n)$. This result follows immediately from Theorem~\ref{Bsla},
together with~\eqref{laiinc}. To compose $B_i$ when they are not ordered as in this result,
one simply uses the result that $B_iB_j=-B_{j-1}B_{i+1}$, which follows routinely from
Theorem~\ref{Bsla} and considering what happens when two right-steps are switched to
up-steps in the two possible orders.

\section{Codes of partitions and the Pl\"ucker relations}\label{sec3}

We consider a set $a_{\cP}=\{a_{\la}:\la\in\cP\}$ of scalars indexed by the set $\cP$ of partitions.
Then the \textit{Pl\"ucker relations} for $a_{\cP}$ are given by the following system
of simultaneous quadratic equations: for
all $m\ge 1$ and  $\al=(\al_1,\ld ,\al_{m-1}),\be=(\be_1,\ld ,\be_{m+1})\in\cP$ with $l(\al)\le m-1$,
 $l(\be)\le m+1$ (which means that $\al_i=0$ for $i> l(\al )$, and $\be_i=0$ for $i> l(\be )$),
we have
\begin{equation}\label{plucker}
\sum_{k=0}^m(-1)^{k-m+1+\ell}
a_{(\al_1-1,\ld,\al_{\ell}-1,\be_{k+1}-k+\ell+1,\al_{\ell +1}, \ld ,\al_{m-1})}
\,\cdot\, a_{(\be_1+1,\ld,\be_k+1,\be_{k+2},\ld ,\be_{m+1})} =0,
\end{equation}
where $\ell=\ell(k)$ is chosen so that $0\le \ell\le m-1$ and
\begin{equation}\label{ellk}
\al_{\ell}-1\ge\be_{k+1}-k+\ell+1\ge\al_{\ell +1},
\end{equation}
with the convention that $\al_0=\infty$, $\al_m=-\infty$,
and so that $\be_{k+1}-k+\ell+1\ge 0$, (if
there is no such choice of $\ell$, then the term in the summation indexed by $k$ is
identically $0$). Note that, for each choice of $k$, then $\ell$ (if it exists) is
unique; to see this, let $\ga_i=\al_i-i$, $i=0,\ld ,m$.
Then $\infty =\ga_0>\ga_1>\cdots >\ga_m=-\infty$, and so for all real numbers $x$, there
is a unique $0\le i\le m-1$ so that $\ga_i>x\ge\ga_{i+1}$. Then, rewriting~\eqref{ellk},
 $\ell$ is the unique choice of $i$ so that (more restrictively so there may not
be such an $i$) $\ga_i-2\ge x\ge\ga_{i+1}$, with $x=\be_{k+1}-k$.

The presentation of the Pl\"ucker relations given in~\eqref{plucker} above 
is referred to as ``classical'' by Fulton~\cite[p.~133]{fu}.
In this presentation,
each equation is a quadratic alternating summation corresponding to an ordered pair of partitions.
Each term in the alternating summation arises from removing a single part from the
second partition, and inserting it into the first partition, with some appropriate
shift in the remaining parts of both partitions. In our next result, we give a different
presentation of the Pl\"ucker relations, which is more symmetrical in its form, using
the notation developed in Section~\ref{sec2} for codes of partitions.

\begin{thm}\label{sympluck}
The \textit{Pl\"ucker relations} for $a_{\cP}$ are given by the following system
of simultaneous quadratic equations: for all $\al,\be\in\cP$,  we have
\begin{equation*}
\sum_{\substack{i,j\ge 1\\|\al^{(i)}|+|\be^{(-j)}|=|\al |+|\be |+1}}
(-1)^{|\al|-|\al^{(i)}|+i+j}a_{\al^{(i)}}\cdot a_{\be^{(-j)}}=0.
\end{equation*}
\end{thm}

\begin{proof}
In the Pl\"ucker relations, equation~\eqref{plucker} is satisfied for each $(m,\al,\be)$ for
 $m\ge 1$ and $\al,\be\in\cP$ with $l(\al)\le m-1$, $l(\be)\le m+1$.
Now multiply~\eqref{plucker} by $(-1)^{m-1}$, to get equation~\eqref{plucker}', and
 consider equation~\eqref{plucker}' for $(m+1,\al,\be)$,
 where we have $\al_m=\be_{m+2}=0$. Then on the left hand side, the term indexed by $k=m+1$ in the latter
equation is identically $0$, since there is no possible choice of $\ell$ (to see
this, we must have $\be_{k+1}-k+\ell+1\ge 0$, so $\ell\ge m$, and since $0\le\ell\le m$, we must uniquely
have $\ell=m$; but then we
have $\al_{\ell}-1=-1<0=\be_{k+1}-k+\ell+1$, contradicting equation~\eqref{ellk}).
Thus equation~\eqref{plucker}' for $(m,\al,\be)$ is identical to
equation~\eqref{plucker}' for $(m+1,\al,\be)$,
so there is the following single equation for
each $\al=(\al_1,\al_2,\ld ),\be=(\be_1,\be_2,\ld )\in\cP$ (which
means that $\al_i=0$ for $i> l(\al )$, and $\be_i=0$ for $i> l(\be )$):
\begin{equation}\label{interlk}
\sum_{k\ge 0}(-1)^{k+\ell}
a_{(\al_1-1,\ld,\al_{\ell}-1,\be_{k+1}-k+\ell+1,\al_{\ell +1}, \ld )}
\,\cdot\, a_{(\be_1+1,\ld,\be_k+1,\be_{k+2},\ld )} =0,
\end{equation}
where $\ell=\ell(k)$ is chosen so that
\begin{equation*}
\al_{\ell}-1\ge\be_{k+1}-k+\ell+1\ge\al_{\ell +1},
\end{equation*}
with the convention that $\al_0=\infty$,
and so that $\be_{k+1}-k+\ell+1\ge 0$.
But, from~\eqref{formisw} and~\eqref{formminisw}, equation~\eqref{interlk} becomes
\begin{equation*}
\sum_{k\ge 0}(-1)^{k+\ell}
a_{\al^{(\be_{k+1}-k+\ell+2)}}\cdot  a_{\be^{(-k-1)}}=0.
\end{equation*}
Finally, note that
\begin{equation*}
\vert \al^{(\be_{k+1}-k+\ell+2)}\vert +\vert \be^{(-k-1)}\vert =
|\al|+|\be|+1,
\end{equation*}
and the result follows from~\eqref{ij},~\eqref{laiinc}, and~\eqref{laminiinc}.
\end{proof}

To fix ideas, we now give a few examples of Pl\"ucker relations. These
examples illustrate that there are redundant equations in the Pl\"ucker relations.
\begin{exa}\label{pluckexs}
For $\al=\be=(1)$, we obtain $\al^{(1)}=\vep$, $\al^{(2)}=(1,1)$, $\al^{(3)}=(2,1)$,
and $\be^{(-1)}=\vep$, $\be^{(-2)}=(2)$, $\be^{(-3)}=(2,1)$, so the corresponding
quadratic equation is
$$-a_{\vep}\cdot a_{(2,1)}+a_{(2,1)}\cdot a_{\vep}=0.$$
But the left hand side of this equation is identically $0$, so the equation is redundant.
\vspace{.1in}

\noindent
For $\al=\vep$, $\be=(1,1,1)$, we obtain $\al^{(1)}=\vep$, $\al^{(2)}=(1)$, $\al^{(3)}=(2)$,
$\al^{(4)}=(3)$,
and $\be^{(-1)}=(1,1)$, $\be^{(-2)}=(2,1)$, $\be^{(-3)}=(2,2)$, so the corresponding
quadratic equation is
\begin{equation}\label{canonpluck}
a_{\vep}\cdot a_{(2,2)}-a_{(1)}\cdot a_{(2,1)}+a_{(2)}\cdot a_{(1,1)}=0.
\end{equation}

\vspace{.1in}

\noindent
For $\al=(2)$, $\be=(1)$, we obtain $\al^{(1)}=(1)$, $\al^{(2)}=(1,1)$, $\al^{(3)}=(2,2)$,
and $\be^{(-1)}=\vep$, $\be^{(-2)}=(2)$, $\be^{(-3)}=(2,1)$, so the corresponding
quadratic equation is
$$-a_{(1)}\cdot a_{(2,1)}+a_{(1,1)}\cdot a_{(2)}+a_{(2,2)}\cdot a_{\vep}=0,$$
which is the same equation as~\eqref{canonpluck}.
\end{exa}

\section{$\tau$-functions for the KP hierarchy and the Pl\"ucker relations}\label{sec5}

There are a number of equivalent descriptions of the KP hierarchy, as
is well described in~\cite{mjd}. The one that we shall start with in this
paper involves the so-called $\tau$-function of the hierarchy, and
two independent sets of indeterminates $p_1,p_2,\ldots$ and $\ph_1,\ph_2,\ldots$. A power series
is said to be a $\tau$-\emph{function for the KP hierarchy} if and only if it
satisfies the bilinear equation
\begin{equation}\label{kpbilinearphys}
[t^{-1}]\exp
\left(\sum_{k\ge 1}\frac{t^k}{k}\left( p_k-\ph_k\right)\right)
\exp
\left(-\sum_{k\ge 1}t^{-k}\left( \frac{\del}{\del p_k}-\frac{\del}{\del \ph_k}\right)\right)
\tau(\bp )\tau(\bph )=0,
\end{equation}
where $\bp=(p_1,p_2,\ldots )$, $\bph=(\ph_1,\ph_2,\ldots )$, and we
use the notation $[A]B$ to mean the \emph{coefficient} of $A$ in $B$.

In considering equation~\eqref{kpbilinearphys}, we regard $p_1,p_2,\ldots$ as the power
sum symmetric functions of an underlying set of variables $x_1,x_2,\ldots$, as in Section~\ref{sec4}.
Also, we regard $\ph_1,\ph_2,\ldots$ as the power
sum symmetric functions of another set of variables, algebraically independent from $x_1,x_2,\ld $
 (one could write this other set of variables as $\xh_1,\xh_2,\ld $, say, but
the precise names of these variables is irrelevant, since no
further explicit mention of either $x_1,x_2,\ld $ or $\xh_1,\xh_2,\ld $ will be made).
We shall use the notation $\eh_i$, $\hh_i$, $\sh_{\la}$, $\Bh(t)$ to denote the
symmetric functions in this other set of variables corresponding
to $e_i$, $h_i$, $s_{\la}$, $B(t)$ in $x_1,x_2,\ldots$.

The following result follows immediately by taking this symmetric function
view of~\eqref{kpbilinearphys}. This view has appeared in a number of
earlier works (see, e.g.,~\cite{jy}), but the results from symmetric functions that
have been applied in these works have been different than ours.
\begin{pro}
A power series is a $\tau$-function for the KP hierarchy
if and only if it satisfies
\begin{equation}\label{BKPform}
[t^{-1}] \bigg( B(t)\tau (\bp)\bigg)\cdot\bigg(\Bh^{\perp}(t^{-1})\tau(\bph )\bigg) =0.
\end{equation}
\end{pro}

\begin{proof}
The result follows immediately from~\eqref{kpbilinearphys},
together with~\eqref{Bernp} and~\eqref{pnperp}.
\end{proof}

Next we give a new proof of the  connection  between Schur
function coefficients of a $\tau$-function for the KP hierarchy, and
the  Pl\"ucker relations. Our proof is immediate from Corollaries~\ref{Bsumsla} and~\ref{Bperpsumsla}.

\begin{thm}[see, e.g.,~{\cite[p. 90,~(10.3)]{mjd}}]\label{SchPlu}
Let the coefficient of the Schur function of shape $\la$ in a power series
be given by $a_{\la}$, $\la\in\cP$. Then the
power  series
is a $\tau$-function for the KP hierarchy if and only if $a_{\cP}=\{ a_{\la}:\la\in\cP\}$ satisfies
the Pl\"ucker relations.
\end{thm}

\begin{proof}
We are given $\tau(\bp)=\sum_{\la\in\cP}a_{\la}s_{\la}$
 and $\tau(\bph)=\sum_{\la\in\cP}a_{\la}\sh_{\la}$.
Then, from~\eqref{BKPform}, it is necessary and
sufficient that $a_{\cP}$ satisfies $S(\bp ,\bph )=0$, where
$$S(\bp ,\bph )=
[t^{-1}] \bigg( B(t)\sum_{\la\in\cP}a_{\la}s_{\la}\bigg)\cdot
\bigg(\Bh^{\perp}(t^{-1})\sum_{\mu\in\cP}a_{\mu}\sh_{\mu}\bigg) .$$
Now, from Corollaries~\ref{Bsumsla} and~\ref{Bperpsumsla}, we immediately obtain
$$S(\bp ,\bph )=
\sum_{\be,\al\in\cP}s_{\be}\, \sh_{\al}
\sum_{\substack{m,k\ge 1\\|\al^{(m)}|+|\be^{(-k)}|=|\al |+|\be |+1}}
(-1)^{|\al|-|\al^{(m)}|+m+k}a_{\al^{(m)}}\cdot a_{\be^{(-k)}}.$$
But $S(\bp ,\bph )=0$ if and only if $[s_{\be}\sh_{\al}]S(\bp ,\bph )=0$ for
all $\be,\al\in\cP$, since  the Schur functions form a basis for symmetric functions,
and the result follows immediately from Theorem~\ref{sympluck}.
\end{proof}

Often the  KP hierarchy is written as a system of simultaneous quadratic partial
differential equations, for $\tau$.
In the next result, we apply Theorem~\ref{SchPlu} and the methods of symmetric functions
to obtain such a system of partial differential equations, with one equation corresponding
to each quadratic equation in the Pl\"ucker relations. The result is well known, but we
include a simple proof for completeness.

\begin{thm}[see, e.g.,~{\cite[p. 92, Lemma 10.2]{mjd}}]\label{pdeplu}
The power series $\tau(\bp)$ is a $\tau$-function for
the KP hierarchy if and only if
the following partial differential equation is satisfied
for each pair of partitions $\al$ and $\be$:
\[ \sum_{\substack{i,j\ge 1\\|\al^{(i)}|+|\be^{(-j)}|=|\al |+|\be |+1}}(-1)^{|\al|-|\al^{(i)}|+i+j}
            \left( s_{\al^{(i)}}(\bp^{\perp})\tau(\bp)\right)\cdot
\left(s_{\be^{(-j)}}(\bp^{\perp})\tau(\bp)\right) = 0. \]
(Where, e.g., $s_{\la}(\bp^{\perp})$ is interpreted as the partial differential
operator obtained by substituting $p_n^{\perp}$ for $p_n$ in~\eqref{psandsp} for each $n\ge 1$,
and using~\eqref{pnperp}.)
\end{thm}

\begin{proof}
Let $\bq = (q_1, q_2, ...)$, where the $q_i$ are independent from the $p_j$ and $\ph_k$. 
We begin the proof by proving that (I): $\tau$ satisfies equation~\eqref{BKPform} if and
only if (II): $\tau$ satisfies
\begin{equation} \label{kpbilinearq}
[t^{-1}] \bigg( B(t)\tau (\bp +\bq )\bigg) 
\cdot
\bigg( \Bh^{\perp}(t^{-1})\tau (\bph +\bq )\bigg)
=0,
\end{equation}
for all $\bq$.

It is easy to see that (II) implies (I), by setting $\bq=\bze$ for $i\ge 1$ (where $\bze =(0,0,\ld )$).

To prove that (I) implies (II),
define the operator $\Ta (\bp)=\exp\left( \sum_{k \geq 1} q_k \frac{\del}{\del p_k} \right)$. 
Using the multivariate Taylor
series expansion of an arbitrary formal power series $f(\bp)$, we see that
\begin{equation}\label{taylor}
f(\bp+\bq) = \Ta (\bp) f(\bp).
\end{equation}
Also define $\Ga (\bp) =\exp\left( \sum_{i \geq 1} \frac{t^i}{i}p_i\right)$, 
 $\Up (\bp)=\exp\left(- \sum_{j \geq 1} t^{-j} \frac{\del}{\del p_j} \right)$,
so $B(t)=\Ga(\bp )\Up(\bp )$.
Then we have
$$
B(t)\tau (\bp +\bq )
=\Ga(\bp )\Up(\bp )\Ta(\bp )\tau(\bp )
=\Ga(\bp )\Ta(\bp )\Up(\bp )\tau(\bp ),$$
from~\eqref{taylor} and the trivial fact that $\Up(\bp )$ commutes with $\Ta(\bp )$.
But using~\eqref{taylor} again, we have the operator identity
$$\Ta(\bp )\Ga(\bp )=\Ga(\bp +\bq)\Ta(\bp )=\Ga(\bq )\Ga(\bp )\Ta(\bp ),$$
and combining these expressions and the fact that $\Ga(\bq )^{-1}=\Ga(-\bq )$ gives
\begin{equation*}
B(t)\tau (\bp +\bq )=\Ga(-\bq )\Ta(\bp )B(t)\tau(\bp ).
\end{equation*}
Similarly we have $\Bh^{\perp}(t^{-1})=\Ga(-\bph )\Up(-\bph )$, and then obtain
\begin{equation*}
\Bh^{\perp}(t^{-1})\tau (\bph +\bq )=\Ga(\bq )\Ta(\bph )\Bh^{\perp}(t^{-1})\tau(\bph ).
\end{equation*}
Multiplying these two expressions together, we find that equation~\eqref{kpbilinearq} becomes
$$
\Ta(\bp )\Ta(\bph )[t^{-1}] \bigg( B(t)\tau (\bp) \bigg)
\cdot
\bigg( \Bh^{\perp}(t^{-1})\tau (\bph )\bigg) =0,$$
since $\Ga(-\bq )\Ga(\bq )=1$, and $\Ta(\bp )$, $\Ta(\bph )$ are independent of $t$.  We conclude
that (I) implies (II).

Finally, in order to apply Theorem~\ref{SchPlu}, we determine the
coefficient of the Schur function of shape~$\la$. This gives
\[ [s_{\la}(\bp )]\tau(\bp +\bq )=  \langle s_{\la}(\bp ), \tau(\bp +\bq )\rangle
 = \langle 1,s_\la(\bp^{\perp}) \tau(\bp+ \bq) \rangle =
\left. s_\la(\bp^{\perp}) \tau(\bp+ \bq)\right|_{\bp =\bze}= 
                                         s_\la(\bq^{\perp}) \tau(\bq), \]
and the result then follows from Theorem~\ref{SchPlu}, replacing~$\bq$ by~$\bp$.
\end{proof}

As an example of Theorem~\ref{pdeplu}, we now give one of the quadratic partial differential equations
for a $\tau$-function.

\begin{exa}
Consider
the  Pl\"ucker equation~\eqref{canonpluck}. Now we have
$$s_{\vep}=1,\quad s_{(1)}=p_1,\quad s_{(2)}=\tfrac{1}{2} (p_1^2+p_2),\quad s_{(1,1)} = \tfrac{1}{2} (p_1^2-p_2),$$
$$s_{(2,1)}=\tfrac{1}{3}(p_1^3-p_3),\quad s_{(2,2)}=\tfrac{1}{12}(p_1^4-4p_1p_3+3p_2^2),$$
so from Theorem~\ref{pdeplu}, the partial differential equation for $\tau$ that corresponds to~\eqref{canonpluck}
is given by
\begin{equation}\label{primplu}
\tfrac{1}{12}\tau\,
\left( \tau_{1111}-12\tau_{13}+12\tau_{22}\right) 
-\tfrac{1}{3}\tau_1\left(\tau_{111}-3\tau_3\right)
+\tfrac{1}{4}\left(\tau_{11}+2\tau_2\right)\left(\tau_{11}-2\tau_2\right) 
=0,
\end{equation}
where we use $\tau_{ijk}$ to denote $\frac{\del}{\del p_i}\frac{\del}{\del p_j}\frac{\del}{\del p_k}\tau$, etc.
\end{exa}

Often, in the literature of integrable systems, the series $F=\log\tau$ is used instead of $\tau$ itself. This series $F$ is 
often referred to as a \emph{solution} to the KP hierarchy, where the ``KP hierarchy'' in 
this context refers to a system of simultaneous partial differential equations for $F$. Of course,
the system of partial differential equations for $\tau$ given in Theorem~\ref{pdeplu} becomes
an equivalent system of partial differential equations for $F$ by substituting $\tau=\exp F$ into
the equations of Theorem~\ref{pdeplu}, and then multiplying the equation by $\exp \left(-2F\right)$. For example,
when we apply this to~\eqref{primplu}, we obtain the equation
$$ \tfrac{1}{12}F_{1111}-F_{13}+F_{22}+\tfrac{1}{2}F_{11}^2 =0,$$
which is often referred to as ``the KP equation''.

\section*{Acknowledgements}

We thank Kevin Purbhoo, Richard Stanley and Ravi Vakil for helpful suggestions.

\end{document}